\documentclass[a4paper,twoside,10pt]{article}
\usepackage[]{geometry}
\pdfoutput=1
\usepackage{hyperref}
\usepackage{url}
\usepackage{authblk}
\usepackage[latin1]{inputenc}
\usepackage{exscale}
\usepackage[centertags]{amsmath}
\usepackage{float}
\usepackage{amsfonts}
\usepackage{amssymb}
\usepackage[pdftex]{graphicx}
\usepackage{ragged2e}
\usepackage{amsthm}
\usepackage{tikz}
\usepackage{tikz-cd}
\usetikzlibrary{calc}
\usetikzlibrary{matrix,arrows,decorations.pathmorphing,decorations.pathreplacing,shapes.misc}\usetikzlibrary{decorations.markings}
\usetikzlibrary{patterns}
\tikzset{
  on each segment/.style={
    decorate,
    decoration={
      show path construction,
      moveto code={},
      lineto code={
        \path [#1]
        (\tikzinputsegmentfirst) -- (\tikzinputsegmentlast);
      },
      curveto code={
        \path [#1] (\tikzinputsegmentfirst)
        .. controls
        (\tikzinputsegmentsupporta) and (\tikzinputsegmentsupportb)
        ..
        (\tikzinputsegmentlast);
      },
      closepath code={
        \path [#1]
        (\tikzinputsegmentfirst) -- (\tikzinputsegmentlast);
      },
    },
  },
  mid arrow/.style={postaction={decorate,decoration={
        markings,
        mark=at position .5 with {\arrow[#1]{stealth}}
      }}},
}
\usepackage{enumerate}
\usepackage{graphics}
\newtheorem{Thm}{Theorem}[section]

\newtheorem{Res*}{}

\newtheorem{Def}{Definition}[section]
\newtheorem{Lemma}{Lemma}[section]

\newtheorem{Prop}{Proposition}[section]

\newtheorem{Rem}{Remark}[section]
\newtheorem{exam}[]{Example}

\providecommand{\keywords}[1]
{
  \small	
  \textbf{\textit{Keywords---}} #1
}

\title{Pointed Gromov-Hausdorff Topological Stability for non-compact metric spaces}

\author[1]{Henry Mauricio S\'anchez Sanabria}
\author[2]{Luis Eduardo Osorio Acevedo}
\affil[1]{Departamento de Matem\'aticas\\ Universidad Central\\
Cra 5  21-38, Bogot\'a Colombia\\
hsanchezs1@ucentral.edu.co}
\affil[2]{Departamento de Matem\'aticas y Estad\'istica\\ Universidad Nacional de Colombia\\Cra 27  64-60, 170003, Manizales, Colombia\\ leosorioa@unal.edu.co}

\date{}

\begin{document}
\maketitle

\abstract{We combine the pointed Gromov-Hausdorff metric \cite{rong10} with the locally $C^0$ distance to obtain the pointed $C^0$-Gromov-Hausdorff distance between maps of possibly different non-compact pointed metric spaces. The latter is then combined with Walters's locally topological stability \cite{lee18} and $GH-$stability from \cite{AM} to obtain the notion of  topologically $GH$-stable pointed homeomorphism. We give one example to show the difference between the distance when take different base point in a pointed metric space.}


\footnote{The second author was supported by Minciencias Colombia grant number 80740-738-2019
and Universidad Nacional de Colombia sede Manizales}

\section{Introduction}
One of the interests in Dynamical Systems theory consists in determine conditions to guarantee the stability of a system, i.e., preservation of the dynamical structure by small perturbations. 
In this way, Walters \cite{walters70} proposed a notion of stability where the behavior of the dynamics $C^0$-close to it is the same as its own dynamic. The notion of topological stability to maps given by Walters, was used by this author to prove that every Anosov diffeomorphism on compact manifolds is topologically stable.
In 2017, Arbieto and Morales \cite{AM} presented an understanding of classical Gromov-Hausdorff metrics extending it to maps between different metric spaces, and started an analysis of asymptotic behavior through the concept of usual topological stability. Then Arbieto and Morales combine the resulting distance $C^0$-Gromov-Hausdorff called with the Walters topological stability to obtain the notion of topological GH-stability. Another approach to this situation was made by other authors, we recommend read Chung \cite{chung20}, Khan-Das-Das \cite{khan18}, Dong-Lee-Morales \cite{dong21}, and references there in, where describe how Gromov-Hausdorff distances was extended to be a distance between maps. 

Since the theory of dynamical systems deal with dynamical properties of group actions, and extension of the distance $C^0$-GH 
to group actions was developed in \cite{dong21}, and several results can be found in  \cite{khan18}. On the other hand, \cite{chulluncuy21} gives a definition of $C^0$-GH distance, that measures distances between two continuous flows of possibly different metric spaces. An important fact in these results was the compactness of the spaces.

Therefore, comes the natural question about how to develop a theory for a non-compact case.  In this way, the result obtained by Walters for Anosov diffeomorphisms is extended for  first countable, locally compact, paracompact, Hausdorff spaces in \cite{das13}, and later \cite{lee18} extend the Walters's stability theorem  to homeomorphisms on locally compact metric spaces.

The purpose behind this work is following to Arbieto and Morales, to provide an understanding of the classical pointed Gromov-Hausdorff distance between non-compact metric spaces (see \cite[p.209]{rong10}) by extending it to maps between proper metric spaces. We call this distance 
the pointed $C^0$-Gromov-Hausdorff distance, and later we combine the resulting distance with the topological stability of locally compact metric spaces \cite{lee18} to obtain the notion of topological pointed GH-stability. The reader will also compared this approach with the definition of pointed Gromov-Hausdorff distance between the two pointed triples and Theorem 3.3 of \cite{rong12}.

The paper is organized as follows:

In Section 2, we state the geometrical and dynamical preliminaries. 
In Section 3, we introduce the pointed $C^0$-Gromov-Hausdorff and we prove some geometrical properties. In Section 4, we establish a definition of topological stability in a pointed sense, we mention our main results and provide some examples and figures that illustrate the concepts.  

\subsection{Acknowledgements}

We thank to Carlos Morales and Seraf{\'i}n Bautista for always be closed to help. The second author is indebted to Stefano Nardulli for all the excellent advice and intuitive explanations of the geometry of non-compact metric spaces, and the encouragement to work with them.
The second author would also to give thanks to Minciencias Colombia for the postdoctoral project 80740-738-2019, and to Universidad Nacional de Colombia sede Manizales its hospitality.


\section{Preliminaries}
\subsection{Gromov-Hausdorff Distance}

Let $(Z,d^Z)$ be a metric space. Let $A\subset Z$, for $\varepsilon>0$, the $\varepsilon$-tube of $A$ is the open set define by $$N_\varepsilon(A)=\left\{z\in Z: d^Z(z,A)<\varepsilon\right\}=\bigcup_{a \in A} B(a, \varepsilon).$$

\begin{Def}[Hausdorff Distance]
 Let $(Z,d)$ be a metric space. Let $A,B$ two bounded subsets of $Z$, for $\varepsilon>0$, 
define  the Hausdorff distance between $A$ and $B$ to be
\begin{align*}
 d_H^Z(A,B)=&\max\left\{ \sup_{a\in A}\inf_{b\in B}d(a,b),
\sup_{b\in B}\inf_{a\in A}d(a,b)
 \right\}\\
 =&\max\left\{ \sup_{a\in A}d(a,B), \sup_{b\in B}d(A,b)\right\} \notag\\
 =& \inf\left\{ \varepsilon>0: A\subset N_\varepsilon(B), B\subset N_\varepsilon(A)  \right\}
\end{align*}
\end{Def}

The Gromov-Hausdorff distance via intrinsic Hausdorff distance is given as:   
\begin{Def}[Intrinsic Gromov-Hausdorff distance] 
Let $X$ and $Y$ be metric
spaces of finite diameter. The Intrinsic Gromov-Hausdorff distance (briefly, $\hat{GH}$-distance)
of $X$ and $Y$ is
$$
\begin{aligned} \hat d_{G H}(X, Y)= \inf _{Z}&\left\{d_{H}^{Z}(\phi(X), \psi(Y)) :\right.\\ 
&\text { $Z$ isometric embedding } \phi : X \hookrightarrow Z, \psi : Y \hookrightarrow Z\} 
\end{aligned}
$$
where $Z$ runs over all such metric spaces, and for each  $Z, \phi$ and $\psi$  run over all possible isometric embedding.
\end{Def}

\begin{Def}
Two maps $\varphi, \psi: X \rightarrow Y$ are $\alpha$-close if and only if  $\forall x \in X$
$$
d^Y(\varphi(x), \psi(x)) \leq \alpha.
$$ 
\end{Def}
 
The following  can be founded by Definition 7.1.4. of \cite{buragoburago}.
\begin{Def}
 Let $(X,d^X)$ and $(Y,d^Y)$ be metric spaces and $i : X \rightarrow Y$ an
arbitrary map. The distortion of $i$ is defined by
$$
\operatorname{dis}(i)= \sup _{x_{1}, x_{2} \in X}\left|d^{Y}(i(x_{1}),i(x_{2}))-d^{X}(x_{1}, x_{2})\right|.
$$
When we required restrict for $U\subset X$, we write
$$
\operatorname{dis}(i)|_U= \sup _{x_{1}, x_{2} \in U}\left|d^{Y}(i(x_{1}),i(x_{2}))-d^{X}(x_{1}, x_{2})\right|.
$$
\end{Def}

For compact metric spaces, $X, Y,$ a map $i: X \rightarrow Y$ is called an $\varepsilon$-GH approximation (briefly,  $\varepsilon$-GHA), if $i$ satisfies the following two conditions:
\begin{enumerate}
\item (it almost preserves distances): $\operatorname{dis}(i)<\varepsilon$.
 \item (it is almost surjective): $N_{\varepsilon}(i(X))=Y$.\\ 
 (i.e, $\forall y\in Y, \exists x \in X; d^Y(i(x), y) \leq \varepsilon$). In particular, $d^Y_{H}(i(X), Y) \leq \varepsilon$.
\end{enumerate}

The set of $\varepsilon$-GHA between $X$ and $Y$ we call $\operatorname{App}_\varepsilon(X,Y)$. We emphasize that an $i\in \operatorname{App}_\varepsilon(X,Y)$ is not necessarily continuous, or injective or surjective.

\begin{Def}
 For compact metric spaces $X$ and $Y$, the following number,
$$
d_{G H}(X, Y)=\inf \{\varepsilon>0 : \exists\ i\in \operatorname{App}_\varepsilon(X,Y) \text { and } j\in\operatorname{App}_\varepsilon(Y,X)\}
$$
is called the Gromov-Hausdorff distance between $X$ and $Y.$
\end{Def}

For preserve the symmetry between $X$ and $Y$ it is necessarily $i$ and $j$, 
but this is not serious, because any $i\in\operatorname{App}_\varepsilon(X,Y)$, exist a approximate inverse, that is, exist $i'\in\operatorname{App}_{4\varepsilon}(Y,X)$
$$
d^{X}\left(i' \circ i(x), x\right) \leq 3 \varepsilon,\quad
d^{Y}\left(i \circ i'(y), y\right) \leq 3 \varepsilon,
$$
Such a map $i^{\prime}$ will be called an $\varepsilon$-inverse of $i$.

A proof of this fact can be founded in \cite[Proposition 1.4]{jansen17} or see \cite{villani08}.

The proof of the followings lemmas  can be found in \cite{rong10} (Lemma 1.3.3 and 1.3.4)
\begin{Lemma}
 Let $X$ and $Y$ be compact metric spaces. If exist $i\in  \operatorname{App}_\varepsilon(X,Y)$, then 
 $d_{G H} \leq 2\hat{d}_{G H} \leq 3 d_{G H}$.
\end{Lemma}

So, we get an estimate for $d_{G H}$ from $\hat{d}_{G H}$, but are different.


\subsection{Pointed Gromov-Hausdorff distance}

For non-compact spaces, the Gromov-Hausdorff distance is not very useful, instead is more convenient work with the concept of  pointed Gromov-Hausdorff distance between pointed metric spaces. We suggest to the reader to see \cite{buragoburago,rong10, jansen17, herron16} and references therein

\begin{Def}
Let $(X, x)$ and $(Y, y)$ be pointed metric spaces, and  let $\varepsilon >0$. 
If there exists  $i:X \to Y$
such that
\begin{enumerate}
\item $i(x)=y$.
 \item $\operatorname{dis}(i)\big|_{B^X(x,1/\varepsilon)}<\varepsilon$.
 \item $B^Y(y,1/\varepsilon-\varepsilon)\subset N_\varepsilon(i(B^X(x,1/\varepsilon)))
 =\bigcup_{y \in i\left(B^X(x,1/\varepsilon)\right)} B^{Y}(y, \varepsilon)
 $.
 \end{enumerate}
For $x \in X$ and $y \in Y$, we call an $\varepsilon$-pointed Gromov-Hausdorff approximation to a function satisfying 1., 2. and 3., and the set of all $\varepsilon$-pointed Gromov-Hausdorff approximations is denoted by
$$
\operatorname{Appr}_{\varepsilon}((X, x),(Y, y)).
$$
\end{Def}
It follows from conditions 1. and 2. that
\begin{equation}
\label{eqball}
i \left(B^X(x, r)\right) \subset B^Y(y, r+\varepsilon)\quad \text{for } 0<\varepsilon<r<\frac1\varepsilon.  \end{equation}

So, (\ref{eqball}) and 3. together give in turn that
$$
B^Y(y, r-\varepsilon) \subset N_{\varepsilon}(i(B^X(x, r))) \subset B^Y(y, r+2 \varepsilon).
$$

\begin{Def}
The pointed GH-distance between $(X,x)$ and $(Y,y)$ is defined by
$$
d_{GH}^{p}((X, p),(Y, q))=
\inf \{\varepsilon>0:
\exists\ i\in \operatorname{App}_\varepsilon((X, x),(Y, y)) \text { and } j\in\operatorname{App}_\varepsilon((Y, y),(X, x))\}.
$$
\end{Def}

\begin{Def}
We say that two pointed metric spaces, $(X, x)$ and $(Y, y),$ are isometric if there is an isometry $h$ from $(X, x)$ to $(Y, y)$,   such that $h(x)=y$.
\end{Def}

Follows as consequence of the definition this result given in  (\cite[Proposition 1.6.3]{rong10}).
\begin{Prop}
 \label{propeqclass}
 $d_{G H}^{p}((X, x),(Y, y))=0$  if and only if $(X, x)$  is isometric to $(Y, y)$.
\end{Prop}

\begin{Rem}
By Proposition \ref{propeqclass}, we work with the isometric classes of pointed complete proper metric spaces.  
\end{Rem}

\begin{Def}[Convergence] 
We say that a sequence $\{(X_k,x_k)\}$  of pointed proper metric space converges to $(X,x)$,
and we write $(X_k,x_k)\stackrel{pGH}{\longrightarrow} (X,x)$, if 
$$\lim_{k \rightarrow \infty} d^p_{GH}\left((X_{k}, x_{k}),(X, x)\right)=0.$$
That is, 
if there is a sequence of  $i_k\in \operatorname{App}_{\varepsilon_k}((X_k, x_k),(X, x))$,  such that $\varepsilon_k\to 0$ as $k\to \infty$. 
\end{Def}

An inspiration to multi-point metric spaces is based in \cite{munoz2020}. 

\begin{Def}
Let $(X, x_1,\dots,x_k)$ and $(Y, y_1,\dots,y_k)$ be multi-pointed metric spaces, and  let $\varepsilon >0$. 
If there exists  $i:X \to Y$
such that
\begin{enumerate}
\item $i(x_n)=y_n$ for $n=1,\dots,k$.
\item $dis(i)|_{B(x_n,1/\varepsilon)}<\varepsilon$, for $n=1,\dots,k$.

 \item $B^Y(y_n,1/\varepsilon-\varepsilon)\subset N_\varepsilon(i(B^X(x_n,1/\varepsilon)))$.
 \end{enumerate}
For $x_n \in X$ and $y_n \in Y$, $n=1,\dots,k$, we call an  $\varepsilon$-multi-pointed Gromov-Hausdorff approximation to a function that 
that satisfies 1,2 and 3, and the set of all $\varepsilon$-multi-pointed Gromov-Hausdorff approximations is denoted by
$$
\operatorname{Appr}_{\varepsilon}((X, x_1,\dots,x_k),(Y, y_1,\dots,y_k)).
$$
\end{Def}

\begin{Def}
A map $f : X \to Y$ between two metric spaces is called $C$-Lipschitz and co-Lipschitz map (briefly, $C$-LcL)
if for any $x \in X$ such that for all $r > 0$, the metric balls satisfy
$$B(f(x),C^{-1}r)\subset  f(B(x,r))\subset B(f(x),Cr).$$
\end{Def}

\subsection{ Gromov-Hausdorff  Convergence of maps}

On \cite{sinaei16} we found the following definition and lemma (Definition 2.11 and Lemma 2.12), compare with Appendix of \cite{grovepetersen91}.

\begin{Def}
Let $\left(X_{i}, x_{i}\right),(X, x),\left(Y_{i}, y_{i}\right)$ and $(Y, y)$ be pointed
metric spaces such that $\left(X_{i}, y_{i}\right)$ converges to $(X, x)$ in the pointed Gromov-Hausdorff topology (resp. $\left(Y_{i}, y_{i}\right)$ converges to $(Y, y)$). We say that a sequence of maps $f_{i} :\left(X_{i}, x_{i}\right) \rightarrow$
$\left(Y_{i}, y_{i}\right)$ converges to a map $f :(X, x) \rightarrow(Y, y)$ if there exists a subsequence $X_{i_{k}}$ such that if
$p_{i_{k}} \in X_{i_{k}}$ and $p_{i_{k}}$ converges to $x$ (in  $\coprod  X_{i_{k}} \coprod X $ with the admissible metric), then  $f_{i_{k}}(p_{i_{k}})$
converges to $f(p) .$
\end{Def}


But to ours purpose we don't use the metric space $\coprod  X_{i_{k}} \coprod X $, so we use an alternative definition for Gromov-Hausdorff convergence of maps, (compare with \cite{wong08} )


\begin{Def}
If $(X_k,x_k) \stackrel{\text { pGH }}{\longrightarrow} (X,x)$, via $i_k \in \operatorname{App}_{\varepsilon_k}((X_k,x_k),(X,x))$,  we say
that points $x_{k} \in X_{k}$ converge to a point $x \in X$  if and only if 
$d^X( i_k(x_{k}), x) \rightarrow 0$. For this convergence we  write  $x_{k} \stackrel{\text { pGH }}{\longrightarrow} x.$
\end{Def}



This permits one to define convergence of maps.
\begin{Def}
If $f_{k}: (X_{k},x_k) \longrightarrow (Y_{k},y_k)$ are maps, $(X_{k},x_k) \stackrel{pG H}{\longrightarrow} (X,x)$ and $(Y_{k},y_k) \stackrel{pG H}{\longrightarrow} (Y,y)$, then we say that $f_{k}$ converge on the sense of pointed Gromov-Hausdorff to a map $f: (X,x) \longrightarrow (Y,y)$ if there exist a subsequence $X_{k_j}$ such that if  $x_{k_j} \stackrel{\text { pGH }}{\longrightarrow} x$, then 
$f_{k_j}\left(x_{k_j}\right) \stackrel{\text { pGH }}{\longrightarrow} f(x)$.
\end{Def}

Proposition B.1. of \cite{wong08}.

\begin{Prop}
 Let $X_{i}$ and $Y_{i}$ be metric spaces with $X_{i} \stackrel{\text { GH }}{\longrightarrow} X$ and $Y_{i} \stackrel{\text { GH }}{\longrightarrow} Y,$ with
$X$ and $Y$ compact. Suppose that for all $i$ there exist $L$-Lipschitz maps $\psi_{i}: Y_{i} \rightarrow X_{i}$.
Then there exists an $L$-Lipschitz map $\psi: Y \rightarrow X .$ If the $\psi_{i}$ are also surjective, then
so is the limit map $\psi$.
\end{Prop}

A family of maps $f_{i} : X_{i} \rightarrow Y_{i}$ is called equicontinuous if, for any $\varepsilon>0,$ there
is $\delta>0$ such that $d_{i}\left(x_{i}, y_{i}\right)<\delta$ implies that $d\left(f_{i}\left(x_{i}\right), f_{i}\left(y_{i}\right)\right)<\varepsilon$ for all $x_{i}, y_{i} \in X_{i}$
and all $i .$

Finally, we mention Lemma 1.6.12 of \cite{rong10}.


\begin{Lemma}[Convergence of maps]\label{Lem:1.6.12Rong}
Let $\left(X_{i}, p_{i}\right) \overset{d_{G H}}{\longrightarrow}(X, p)$ and $\left(Y_{i}, q_{i}\right) \overset{d_{G H}}{\longrightarrow }(Y, q)$, and let $f_{i} :\left(X_{i}, p_{i}\right) \rightarrow\left(Y_{i}, q_{i}\right)$ be a continuous.
\begin{enumerate}
\item  
If $f_{i}$ are equicontinuous, then there is a uniform continuous map
$f :(X, p) \rightarrow(Y, q)$ and a subsequence $X_{i_{k}}$ such that if $x_{i_{k}} \in X_{i_{k}}$, $x_{i_{k}} \rightarrow x,$ then
$f_{i_{k}}\left(x_{i_{k}}\right) \rightarrow f(x)$.  (We then say $f_{i} \rightarrow f.$)
\item If $f_{i}$ are isometries, then the limit map $f :(X, p) \rightarrow(Y, q)$ is also
an isometry.
\end{enumerate}
\end{Lemma}


\subsection{Local Stability}


Not always it is possible to modify the definition of a global property to a local property only paraphrasing it as ``in some neighbourhood of a given point'' \cite{irwin2001}. Our problem it is to find a good and consistent definition to have a comparison of two non-compact dynamical systems that are close one to other in the pointed Gromov-Hausdorff sense, and the behavior of his homeomorphisms are close in the sense of $C^0$-topology locally. In spite of these observations we find the straightforwardly adapted definitions useful.


\begin{Def}
Let $f: X \rightarrow X$ and $g: Y \rightarrow Y$ be homeomorphisms of topological spaces $X$
and $Y .$ A topological conjugacy from $f$ to $g$ is a homeomorphism $h: X \rightarrow Y$
such that $h\circ f=g\circ h$.
\end{Def}


\begin{Def}
 Let $U$ and $V$ be open subsets of metric spaces $X$ and $Y$ respectively,
and let $f: X \rightarrow X$ and $g: Y \rightarrow Y$. We say that $f | U$ is 
locally isometric (resp. locally homeomorphic)
 to $g | V$ if there exists an isometry (resp. homeomorphism)
 $h: U \cup f(U) \rightarrow V \cup g(V)$ such that $h(U)=V$ and, for all $x \in U,$
$h\circ f(x)=g\circ h(x)$. 

We say that $f$ is locally isometric (resp. locally homeomorphic)
to $g$ if for all open $U\subset X$, and $V\subset Y$ such a map $h$ exists.

\end{Def}

\subsection{$C^0$-Gromov-Hausdorff}
A slight modification of the Gromov-Hausdorff distance including the $C^{0}$-distance
above yields the following definition.

To introduce a metric in the space of maps of
metric spaces we recall the classical $C^{0}$-distance between the maps.

\begin{Def}
Consider $f,g : X \rightarrow Y$, define the $C^{0}$-distance between $f$ and $g$ by
$$
d_{C^{0}}(f, g)=\sup _{x \in X} d^{Y}(f(x), g(x)).
$$
\end{Def}

Now with the notion of GH-distance, we want compare maps of different metric spaces.
Following Arbieto-Morales (See Definition 1.1 of \cite{AM}).
\begin{Def}
Let $X,Y$ be compact metric spaces. 
 The $C^{0}$-Gromov-Hausdorff distance between maps $f : X \rightarrow X$
and $g : Y \rightarrow Y$ is defined by
$$
\begin{aligned} d_{G H^{0}}(f, g)=& \inf \{\varepsilon>0 : \exists\ i\in \operatorname{App}_\varepsilon(X,Y) \text { and } j\in\operatorname{App}_\varepsilon(Y,X) \text { such that }\\ &\left.d_{C^{0}}(g \circ i, i \circ f)<\varepsilon \text { and } d_{C^{0}}(j \circ g, f \circ j)<\varepsilon\right\}. \end{aligned}
$$
\end{Def}


\begin{Rem}
\label{strong GH0}
 It is easy to prove that if $Y=\{y\}$, then $d_{GH}(X,Y)=\operatorname{diam}(X)$, if we take $f$ and $g$ such as the respective identities we have that $d_{GH^0}(f,g)=\operatorname{diam}(X)$.
\end{Rem} 

The basic properties of the map $d_{GH0}$ are defined below (Theorem 1 of \cite{AM}).

\begin{Thm}\label{Thm:AM1}

 Let $(X,d^X)$ and $(Y,d^Y)$ be compact metric spaces. Let  $f : X \rightarrow X$ and $g : Y \rightarrow Y$ be maps, then
\begin{enumerate}
\item If $X=Y,$ then $d_{G H^{0}}(f, g) \leq d_{C^{0}}(f, g)$ (and the equality is not necessarily
true).
\item $d_{G H}(X, Y) \leq d_{G H^{0}}(f, g)$ and $d_{G H}(X, Y)=d_{G H^{0}}\left(I d_{X}, I d_{Y}\right)$ where $I d_{Z}$ is the identity map of $Z$.
\item  If $X$ and $Y$ are compact and $g$ continuous, then $d_{G H^{0}}(f, g)=0$ if and only if
$f$ and $g$ are isometric.
\item  $d_{G H^{0}}(f, g)=d_{G H^{0}}(g, f)$
\item For any map $r : Z \rightarrow Z$ of any metric space $Z$ one has
$$
d_{G H^{0}}(f, g) \leq 2\left(d_{G H^{0}}(f, r)+d_{G H^{0}}(r, g)\right).
$$
\item $d_{G H^{0}}(f, g) \geq 0$ and if $X$ and $Y$ are bounded, then $d_{G H^{0}}(f, g)<\infty$.
\item If $X$ is compact and there is a sequence of isometries $g_{n} : Y_{n} \rightarrow Y_{n}$ such that
$d_{G H^{0}}\left(f, g_{n}\right) \rightarrow 0$ as $n \rightarrow \infty,$ then $f$ is also an isometry.
\end{enumerate}
\end{Thm}

Note that in particular
$d_{G H^0}(f,g)<\varepsilon$ implies that $\hat d_{G H}(X,Y)<\frac{3}{2} \varepsilon$.

%
%

\section{Pointed $C^0$-Gromov-Hausdorff}

\begin{Def}[Pointed $C^0$-Gromov-Hausdorff]
 Let $(X,x)$, $(Y,y)$ be pointed metric spaces. 
 The Pointed $C^{0}$-Gromov-Hausdorff distance between maps $f : X \rightarrow X$
and $g : Y \rightarrow Y$ is defined by
$$
\begin{aligned} d^p_{G H^{0}}((X,x,f),(Y,y,g))= \inf \{\varepsilon>0 &: \exists\ 
i\in \operatorname{App}_\varepsilon((X,x,f(x)),(Y,y,g(y))) \text { and }\\
&j\in \operatorname{App}_\varepsilon((Y,y,g(y)),(X,x,f(x)))
\text { such that }\\
&\left.
d_{0,\varepsilon,p}(i\circ f,g\circ i)<\varepsilon, d_{0,\varepsilon,p}(j\circ g,f\circ j)<\varepsilon
\right\},
\end{aligned}
$$
where we call 
\begin{align*}
d_{0,\varepsilon,p}(k\circ f,g\circ k)&=\sup_{q\in B(x,1/\varepsilon)} d^{B(g(y),1/\varepsilon)}(k\circ f(q),g\circ k(q)).
\end{align*}
By abuse of notation and if don't required explicit the pointed spaces we write $d^p_{G H^{0}}((X,x,f),(Y,y,g))=d^p_{G H^{0}}(f,g)$.
\end{Def}

The existence of the $\varepsilon$-approximations $i,j$ is a requirement to get the symmetry, but if we guarantee the existence of the  $\varepsilon$-approximations $i$ then exist $4\varepsilon$-approximation $i'$ that  satisfies the conditions of $j$. (The pointed version of the observation made above).


Note that when restricting to compact metric spaces, the $d_{G H^0}$-convergence is stronger than the $d_{G H^0}^{p}$-convergence. It is a consequence of Remark \ref{strong GH0}. On the other hand, 
 let $(X_{n}, p_{n})=(\{0, n\}, 0)$ and $(X,0)=(\{0\},0)$ be pointed metric spaces. We have that $d_{GH}^p\left( (X_{n}, p_{n}),(X, 0)\right)=\frac{1}{n}$, while $X_{n}$ does not converge in $d_{G H}$ to $X$. Then if $f\ne Id$ and $g=Id$, for $\frac1\varepsilon <n$, we have that 
$$d_{G H^{0}}^p\left((\{0,n\},0,f),(\{0\},0,g)\right)=\frac{1}{n}.$$

Therefore, $d_{GH^0}^p$-convergence no implies $d_{GH^0}$-convergence.

Now, with this concept, we give the pointed version of Theorem \ref{Thm:AM1}. 



\begin{Thm}
 Let $(X,x)$ and $(Y,y)$ proper pointed metric spaces. Let  $f : X \rightarrow X$ and $g : Y \rightarrow Y$ maps, then
\begin{enumerate}
\item If $(X,x)=(Y,y),$ then $d^p_{G H^{0}}(f,g) \leq d_{0,\varepsilon}^{x,y}(f,g)$ (and the equality is not necessarily
true).
\item $d^p_{G H}((X,x), (Y,y)) \leq d_{G H^{0}}^p(f, g)$ and $d^p_{G H}((X,x), (Y,y))=d^p_{G H^{0}}\left(I d_{X}, I d_{Y}\right)$ where $I d_{Z}$ is the identity map of $Z$.
\item  $d^p_{G H^{0}}(f, g)=0$ if and only if 
$f$ and $g$ are locally isometric.
\item  $d^p_{G H^{0}}(f, g)=d^p_{G H^{0}}(g, f)$.
\item For $f_m : X_m \rightarrow X_m$ with $f_m$ $C_m$-Lipschitz, $m=1,2,3$. We have that
$$
d^p_{G H^{0}}(f_1, f_3) \leq 2\left(d^p_{G H^{0}}(f_1, f_2)+d^p_{G H^{0}}(f_2, f_3)\right).
$$
\item $0\le d^p_{G H^{0}}(f, g)<\infty$ .
\item If $(X,x)$ is a proper pointed metric space and there is a sequence of isometries $g_{n} : (Y_{n},y_n) \rightarrow (Y_{n},g(y_n))$ such that $d_{G H^{0}}^p\left(f, g_{n}\right) \rightarrow 0$ as $n \rightarrow \infty,$ then $f$ is also an isometry.
\end{enumerate}
\end{Thm}

\begin{proof}
\begin{enumerate}
\item If $(X,x)=(Y,y)$. Take $\delta>0$, fix small $\varepsilon_0$, and let $\varepsilon=d_{0,\varepsilon_0}^{x,y}(f,g)+\delta$ be, $i=j=Id_{X}$.
Then  $i,j\in \operatorname{Appr}_{\varepsilon_0}((X,x,x),(X,x,x))$. Therefore 
\begin{align*}
d_{0,\varepsilon_0}^{x,y}(f,g)
&=\sup_{p\in B(1/\varepsilon_0,x)}d^{B(g(y),1/\varepsilon_0)}(i\circ f(p),g\circ i(p))\\
&=\sup_{p\in B(1/\varepsilon_0,x)}d^{B(g(y),1/\varepsilon_0)} ( f(p),g(p))\\
&<\varepsilon.
\end{align*}

Therefore, $d^p_{G H^{0}}((X,x,f),(Y,y,g)) \leq \varepsilon=d_{0,\varepsilon_0}^p(f,g)+\delta$, since $\delta$ is arbitrary, $d^p_{G H^{0}}(f, g) \leq d^p_{0,\varepsilon_0}(f, g) .$

If we take to locally isometric maps $f$ and $g$, but no equal, we get $d^p_{G H^{0}}(f, g)=0$ but 
$d_{0,\varepsilon}^{x,y}(f,g)>0$.


\item  It follows easily from the definitions that $d^p_{G H}((X,x), (Y,y)) \leq d^p_{G H^{0}}(f, g)$.  
Suppose that $d^p_{G H}((X,x), (Y,y))<\varepsilon$, then fix $\delta>0$, such that $\varepsilon<d^p_{G H}((X,x), (Y,y))+\delta$, and 
 $i\in \operatorname{App}_\varepsilon((X,x,f(x)),(Y,y,g(y)))$, 
 $j\in \operatorname{App}_\varepsilon((Y,y,g(y)),(X,x,f(x)))$. On the other hand 
\begin{align*}
d_{0,\varepsilon_0}^{x,y}(Id_X,Id_Y)
&=\sup_{p\in B(1/\varepsilon,x)}d^{B(y,1/\varepsilon)}(i\circ Id_X(p),Id_Y\circ i(p))\\ 
&=0< \varepsilon,
\end{align*}
and similarly $d_{0,\varepsilon}^{y,x}(Id_Y,Id_X)=0<\varepsilon$.
Then  $d_{G H^{0}}^p\left(I d_{X}, I d_{Y}\right) \leq \varepsilon<d_{G H}((X,x), (Y,y))+\delta$. As $\delta$ is arbitrary, $d_{G H^{0}}^p\left(I d_{X}, I d_{Y}\right) \leq d_{G H}((X,x), (Y,y)) .$ From this we conclude that
$d_{G H^{0}}^p\left(I d_{X}, I d_{Y}\right)=d_{G H}((X,x), (Y,y)) .$

\item 

If $f$ and $g$ are locally isometric, then
there is an locally isometry $h : U\cup f(U) \rightarrow V\cup g(V)$ for all open $U\subset X$, and $V\subset Y$ such that $g \circ h(p)=h \circ f(p)$ for all $p\in U$. Hence, for any $\varepsilon>0$,
$h : B(1/\varepsilon,x)\cup f(B(1/\varepsilon,x))\to B(1/\varepsilon,y)\cup g(B(1/\varepsilon,y))$
and $h^{-1}$ are locally isometries, in particular  $h\in \operatorname{App}_{\varepsilon}((X,x,f(x)),(Y,y,g(y))$ and $h^{-1}\in \operatorname{App}_{\varepsilon}((Y,y,g(y),(X,x,f(x)))$
satisfying
\begin{align*}
\sup_{p\in B(1/\varepsilon,x)}d^{B(g(y),1/\varepsilon)}
(h\circ f(p), g\circ h(p))=&
\sup_{q\in B(1/\varepsilon,y)}d^{B(f(x),1/\varepsilon)}
(g\circ h^{-1}(p), h^{-1}\circ f(p))\\
=&0<\varepsilon.
\end{align*}
Hence $d^p_{G H^{0}}(f, g) \leq \varepsilon.$ As $\varepsilon$ is arbitrary, $d^p_{G H^{0}}(f, g)=0.$

On the other hand, suppose that 
$d^p_{G H^{0}}(f, g)=0$,  then there $i_n\in \operatorname{App}_{1/n} ((X, x,f(x)),(Y, y,g(y)))$,  and $j_n\in \operatorname{App}_{1/n} ((Y, y,g(y)),(X, x,f(x)))$, such as
$$
d_{0,\frac1n}^{x,y}(f,g)
=\sup_{p\in B(x,n)}d^{B(g(y),n)}(i_n\circ f(p) ,g\circ i_n(p))<\frac{1}n
$$
and
$$
d_{0,\frac1n}^{y,x}(g,f)
=\sup_{q\in B(y,n)}d^{B(f(x),n)}(j_n\circ g(q) ,f\circ j_n(q))<\frac{1}n.
$$
We can extend $i_n$ to a map from $X$ to $Y$ as a $\operatorname{App}_{2/n} ((X,x,f(x)),(Y,y,g(y))$, and similarly $j_n$. Because $X$ and $Y$ are
proper (closed balls are compact). Therefore the sequence
of maps $i_n$, $j_n$ has a subsequence that converges uniformly on compact sets to an
isometry from $X$ to $Y$ satisfying the conditions.

\item Follows from  the definition.
\item Now fix $\delta>0$, and from definition, there exist $i_{mn}\in \operatorname{App}_{\varepsilon_{mn}}((X_m,x_m,f_m(x_m)),(X_n,x_n,f_n(x_n))$, with
$$\varepsilon_{mn}< d^p_{GH0}(f_m,f_n)+\delta,$$
and such that

$$
d^{x_mx_n}_{0,\varepsilon_{mn}}(f_m,f_n)=\sup_{B^m(x_m,\varepsilon_{mn}^{-1})} d^{B^n(f_n(x_n),\varepsilon_{mn}^{-1})}( i_{mn}\circ f_n,f_m\circ i_{mn})<\varepsilon_{mn}
$$
for $(m,n)$ equal to $(1,2)$, $(2,3)$, $(2,1)$ or $(3,2)$, and $\varepsilon_{mn}=\varepsilon_{nm}<\frac12$.

Similar to the proof of triangle inequality (\ref{p:tinq1}), we can show that $$i_{kl}\in\operatorname{App}_{\varepsilon_{kl}}((X_k,x_k,f_k(x_k)),(X_l,x_l,f_l(x_l)),$$
for $i_{kl}=i_{k2}\circ i_{2l}$
for $(k,l)=(1,3)$ or $(3,1)$, and 
$\varepsilon_{kl}=\varepsilon_{lk}=\min\{2(\varepsilon_{k2}+\varepsilon_{2l}),\frac12\}$.

Observe that if $f_k$ is $C_k$-Lipschitz for $k=1,2$ then

$f_1\left(B^1(x_1,\varepsilon_{13}^{-1})\right)\subset B^1(f_1(x_1),C_1\varepsilon_{13}^{-1})$. So, then $i_{12}\left[ f_1\left(B^1(x_1,\varepsilon_{13}^{-1})\right)\right]
\subset B^2(f_2(x_2),C_1\varepsilon_{13}^{-1}+\varepsilon_{12})
\subset B^2(f_2(x_2),\varepsilon_{23}^{-1})
$ and 
$i_{12}\left(B^1(x_1,\varepsilon_{13}^{-1}) \right)\subset B^2(x_2,\varepsilon_{13}^{-1}+\varepsilon_{12})
$. Therefore, 

$$f_2\left[ i_{12}\left(B^1(x_1,\varepsilon_{13}^{-1}) \right) \right]
\subset B^2(f_2(x_2),C_2(\varepsilon_{13}^{-1}+\varepsilon_{12}))
\subset B^2(f_2(x_2),\varepsilon_{23}^{-1}).
$$

Then
$$
\begin{aligned}
d^3&( i_{13}\circ f_1(p_1), f_3\circ i_{13}(p_1))\\ 
&\le 
d^3( i_{13}\circ f_1(p_1), i_{23}\circ f_2\circ i_{12}(p_1)) +
d^3( i_{23}\circ f_2\circ i_{12}(p_1), f_3\circ i_{13}(p_1)) \\
&= d^3( i_{23}\left(i_{12}\circ f_1(p_1)\right), i_{23}\left( f_2\circ i_{12}(p_1)\right))+
d^3( i_{23}\circ f_2\left( i_{12}(p_1)\right), f_3\circ i_{23}\left(i_{12}(p_1)\right)) \\
&\le d^2\left(i_{12}\circ f_1(p_1),  f_2\circ i_{12}(p_1)\right)+\varepsilon_{23}+\varepsilon_{23}\\
&\le \varepsilon_{12}+2\varepsilon_{23}< \varepsilon_{13}.
\end{aligned}
$$
Now, we have  
$$
d^{x_1x_3}_{0,\varepsilon_{13}}(f_1,f_3)=
\sup_{p_1\in B^1(x_1,\varepsilon_{13}^{-1})} 
d^{ B^3(f_3(x_3),\varepsilon_{13}^{-1}) }(
i_{23}\circ i_{12}f_1(p_1), f_3\circ i_{23}\circ i_{12})<
\varepsilon_{13}$$
and similarly
$$
d^{x_3x_1}_{0,\varepsilon_{31}}(f_3,f_1)<\varepsilon_{31}.
$$
And this implies that
$$
d^p_{GH^0}(f_1,f_3)\le   2(\varepsilon_{12}+\varepsilon_{23})
\le 2(d^p_{GH^0}(f_1,f_2)+d^p_{GH^0}(f_2,f_3)+2\delta),
$$
and since $\delta$ is arbitrary  we obtain
$$
d^p_{GH^0}(f_1,f_3)\le   2\left(d^p_{GH^0}(f_1,f_2)+d^p_{GH^0}(f_2,f_3)\right).
$$

\item Follows immediately by definition and by the localness of the distance.
\justifying{
\item Since $d_{GH^0}^p(f,g_n)\to 0$, then for any $r>0$, we can assume that $d_{GH0}^p(f|_{B(x,r)},g_n|_{B(y_n,r)})\to 0$, applying the Theorem \ref{Thm:AM1}.7, then $f|_{B(x,r)}=f_r$ is a isometry. Let $X=\bigcup_rB(x,r)$, clearly we can assume that $B(x,r)\subset B(x,R)$ for $r<R$, then $f_r=f_R|_{B(x,r)}$. Then we have that $f$ is a isometry.}
\end{enumerate}
\end{proof}

\section{Topological stability}

\begin{Def}
A homeomorphism $f: X \rightarrow X$ of a compact metric space $X$ is topologically stable if for every $\varepsilon>0$ there is $\delta>0$ such that for every homeomorphism $g: X \rightarrow X$ with $d_{C^{0}}(f, g)<\delta$ there is a continuous map $h: X \rightarrow X$ with
$$
d_{C^{0}}\left(h, I d_{X}\right)<\varepsilon \text { such that } f \circ h=h \circ g
.$$
\end{Def}

\begin{Def}
A homeomorphism $f: X \rightarrow X$ of a compact metric space $X$ is topologically $GH$-stable if for every $\varepsilon>0$ there is $\delta>0$ such that for every homeomorphism $g: Y \rightarrow Y$ of compact metric space $Y$ satisfying  $d_{GH^{0}}(f, g)<\delta$ there is a continuous map $h\in App_\varepsilon(Y,X)$ such that $f \circ h=h \circ g$.
\end{Def}
The below result gives sufficients conditions for $GH$-stability (Theorem 4 \cite{AM}).
\begin{Thm}
 Every expansive homeomorphism with the POTP of a compact metric space is topologically $GH$-stable.
\end{Thm}

Following \cite{lee18}, we give the notions of expansiveness, shadowing property and topological stability for homeomorphisms on non-compact metric spaces, that their are dynamical properties.  We denote $\mathcal{C}(X)$ the collection of continuous functions from $X$ to $(0, \infty)$.

\begin{Def}
 Let $X$ be a metrizable space and $f$ be a homeomorphism of $X$ onto itself. We say that
 
\begin{enumerate}
\item $f$ is $\mathcal{C}$-expansive if there exists a metric $d$ for $X$ and $\delta \in \mathcal{C}(X)$ such that $d\left(f^{n}(x), f^{n}(y)\right)<\delta\left(f^{n}(x)\right)(x, y \in X)$ for all $n \in \mathbb{Z}$ implies $x=y$.
\item $f$ has the $\mathcal{C}$-shadowing property if for any $\varepsilon \in \mathcal{C}(X),$ there exists $\delta \in \mathcal{C}(X)$ such that for any $\delta$-pseudo orbit $\left\{x_{n}\right\}_{n \in \mathbb{Z}}$ of $f ;$ i.e., $d\left(f\left(x_{n}\right), x_{n+1}\right)<\delta\left(f\left(x_{n}\right)\right)$ there is a point $x \in X$ which $\varepsilon$-traces the pseudo orbit, that is, $d\left(f^{n}(x), x_{n}\right)<$ $\varepsilon\left(f^{n}(x)\right)$ for all $n \in \mathbb{Z}$.
\item $f$ is $\mathcal{C}$-topologically stable if for any $\varepsilon \in \mathcal{C}(X),$ there is $\delta \in \mathcal{C}(X)$ such that if $g$ is any homeomorphism of $X$ with $d(f(x), g(x))<\delta(f(x))$ for all $x \in X,$ then there is a continuous map $h: X \rightarrow X$ with $h \circ g=f \circ h$ and $d(h(x), x)<\varepsilon(h(x))$ for all $x \in X$.
\end{enumerate}
\end{Def}

Some technical necessary results are the following

\begin{Lemma}[Lemma 2.8 of \cite{lee18}]
\label{lem:2.8lee}
 For any $\alpha \in \mathcal{C}(X)$, there is $\gamma\in \mathcal{C}(X)$ such that 
 $\gamma(x) < \inf\{\alpha(y): y\in B(x, \gamma(x))\}$ for all $x\in X$.
\end{Lemma}

\begin{Rem}[Remark 3.2 of \cite{lee18}]
\label{rem:3.2lee}
 For any $\varepsilon \in \mathcal{C}(X)$ there exists $\gamma \in \mathcal{C}(X)$ such that $\gamma(x) < \inf\{\varepsilon(y): y \in B(x, \gamma(x))\}$ for $x \in X$, by Lemma \ref{lem:2.8lee}. Hence, if $d(x, y) < \max\{\gamma(x), \gamma(y)\}$ $(x, y \in X)$, then $d(x, y) < \varepsilon(x)$. Indeed, if $\max\{\gamma(x), \gamma(y)\} = \gamma(x)$, then $d(x, y) < \gamma(x) < \varepsilon(x)$; otherwise, we have $x\in B(y, \gamma(y))$ and so $\gamma(y) < \varepsilon(x)$. Finally we get $d(x, y) < \varepsilon(x)$.
 \end{Rem}

\begin{Lemma}[Lemma 3.3 of \cite{lee18}]
\label{lem:3.3lee}
Let $f$ be an expansive homeomorphism of a locally compact metric space $X$ with an expansive function $e \in \mathcal{C}(X)$ with $e < \alpha_X$, where $\alpha_X\in\mathcal{C}(X)$ such that $\overline{B(x,\alpha_X(x))}$ is compact. For any $x_0 \in X$ and $\lambda \in \mathcal{C}(X)$, there is $N > 0$ such that if $d(x_0,y) \ge \lambda(x_0)$, then $d(f^n(x_0),f^n(y))) \ge e(f^n(x_0))$ for some $|n| < N$.
\end{Lemma}

In \cite[Theorem 1.1]{lee18} they extend the Walters's stability theorem  to homeomorphisms on locally compact spaces. 
\begin{Thm}
Let $X$ be a locally compact metric space and $f$ be a homeomorphism of $X$ onto itself. If $f$ is $\mathcal{C}$-expansive and has the $\mathcal{C}$-shadowing property then it is $\mathcal{C}$-topologically stable.
\end{Thm}

\begin{Rem}
The $\mathcal{C}$-shadowing implies the POTP considering $\varepsilon \in \mathcal{C}(X)$ and  $\delta \in \mathcal{C}(X)$ as constants functions.
\end{Rem}

Now we can give a definition for spaces topologically $GH$-stable in a local pointed sense.

\begin{Def}
A homeomorphism $f: (X,x) \rightarrow (X,f(x))$ of a proper metric space $X$ is topologically $pGH$-stable if for every $\varepsilon\in\mathcal{C}(X)$,  there is $\delta\in\mathcal{C}(X)$ such that for every homeomorphism $g: Y \rightarrow Y$ of metric space $Y$ satisfying  $d^p_{GH^{0}}(f, g)<\delta(f(x))$ there is a continuous map $h\in App_\varepsilon(Y,X)$ such that $f \circ h=h \circ g$.
\end{Def}

\begin{Thm}
\label{mainstatheorem}
Let $(X,x)$ be a proper pointed metric space and $f$ be a homeomorphism of $X$ onto itself. If $f$ is $\mathcal{C}$-expansive and has the $\mathcal{C}$-shadowing property then it is topologically $pGH$-stable.
\end{Thm}
\begin{proof}
Let $f:X\to X$ be expansive with an $\mathcal{C}$-expansive function $e \in \mathcal{C}(X)$, 
let  $\gamma \in \mathcal{C}(X)$ be such that $\gamma(z) < \inf\{e(y)| y \in B(z,\gamma(z))\}$ for all $z\in X$.

Fix $\varepsilon > 0$ and take $0 < \bar{\varepsilon} < \frac 14\min\{\varepsilon, \gamma\}$. 

For this $ \bar{\varepsilon}$  we choose $\delta$ from the POTP. We can assume that $\delta < \bar\varepsilon$.

First we claim that any $\delta$-pseudo orbit is $\bar\varepsilon$-traced by only one point in $X$. Indeed, we suppose that there are 
$p,q \in X$ which $\bar\varepsilon$-trace a $\delta$-pseudo orbit $(p_n)_{n\in \mathbb Z}$. Then we get
$$
\begin{aligned}
d\left(f^{n}(p), f^{n}(q)\right) & \leq d\left(f^{n}(p), p_{n}\right)+d\left(p_{n}, f^{n}(q)\right)<\bar\varepsilon\left(f^{n}(p)\right)+\bar\varepsilon\left(f^{n}(q)\right) \\
&<\frac{\gamma\left(f^{n}(p)\right)}{4}+\frac{\gamma\left(f^{n}(q)\right)}{4}<\max \left\{\gamma\left(f^{n}(p)\right), \gamma\left(f^{n}(q)\right)\right\} \\
&<e\left(f^{n}(p)\right) \quad \text { by Remark \ref{rem:3.2lee}}
\end{aligned}
$$
for all $n \in \mathbb{Z}$. So, the expansivity implies $p=q$.

Now, let $g$ be a homeomorphism of $Y$ with $d^p_{GH^0}(f, g)<\delta$ for all $x \in X$.

Then, there exists $j\in App_\delta((Y,y,g(y)),(X,x,f(x)))$ such that $d_{0,\delta,p} (j \circ g, f \circ j) < \delta$.

Take $q \in B(y,1/\delta)\subset Y$ and consider the sequence $(p_n)_{n\in\mathbb Z}$ defined by $p_n = j(g^n(q))$ for $n \in\mathbb Z$, and choose $\delta(g^n(y))$ such that $p_n\in B(f(x),\frac1\delta)$. Since
$$
\begin{aligned}
 d^{B(f(x),1/\delta)}(p_{n+1}, f(p_n)) 
 &= d^{B(f(x),1/\delta)}(j [g^{n+1}(q)], f (j[g^n(q)])) \\
 &= d^{B(f(x),1/\delta)} (j \circ g(g^n(q)), f \circ j(g^n(q))) < \delta
\end{aligned}
$$
for all $n\in \mathbb N$.

Then $(j\circ g^{n}(q))_{n \in \mathbb{Z}}$ is a $\delta$-pseudo orbit of $f$ for each $q \in B(y,1/\delta)\subset Y$. Hence we can define a map $h: B(y,1/\delta) \rightarrow X$ by $h(q)=$ the unique shadowing point of the $\delta$-pseudo orbit $(j\circ g^{n}(q))_{n \in \mathbb{Z}}$. So, we have
$$
d^{B(f(x),1/\delta)}(\left(f^{n}(h(q)), j(g^{n}(q))\right)<\bar \varepsilon\left(f^{n}(h(q))\right) 
$$
for all  $q\in B(y,1/\delta)\subset Y$ and $n\in\mathbb Z$.

Taking $n=0$ above we get $d^{B(f(x),1/\delta)}\left(h(q), j(q)\right)<\bar \varepsilon$ for all  $q\in B(y,1/\delta)\subset Y$.
Then $d^{B(f(x),1/\delta)}_{0}(h,j)<\varepsilon$.

Since $j\in App_\delta$, we have that for all $p\in B(f(x),1/\delta)\subset X$ we have that
$$
d ( h(B(y,1/\delta)),p) \le d^{B(f(x),1/\delta)}_{0}(h,j)+d^{B(f(x),1/\delta)}(j(B(y,1/\delta)),p)<\bar\varepsilon+\delta<\varepsilon,
$$
and for all $p_1,p_2\in B(y,1/\delta)$

$$
\begin{aligned}
&\left|d^X (h(p_1), h(p_2))-d^Y(p_1, p_2)\right|\\
&\leq \left|d^{X} (h(p_1), h(p_2))-d^{X}(j(p_1), j(p_2))\right|
+\left|d^{X}(j(p_1), j(p_2))-d^{Y}(p_1, p_2)\right| \\ 
&\leq \left|d^{X}\left(h(p_1), h\left(p_2\right)\right)-d^{X}\left(h(p_1), j\left(p_2\right)\right)\right|+
\left|d^{X}\left(h(p_1), j\left(p_2\right)\right)-d^{X}\left(j(p_1), j\left(p_2\right)\right)\right|+\delta \\ 
&\leq d^{X}\left(h\left(p_2\right), j\left(p_2\right)\right)+d^{X}(h(p_1), j(p_1))+\delta \\
&< 2 \bar{\varepsilon}+\delta < \frac{\varepsilon}{4}+\frac{\varepsilon}{8} < \varepsilon.
\end{aligned}
$$
Then $dis(h)\big|_{B(y,1/\delta)}<\varepsilon$, and this implies $dis(h)\big|_{B(y,1/\varepsilon)}<\varepsilon$, that is, $h\in App_\varepsilon$.

On the other hand, since
$$
d(f^n(h(g(q))), j(g^n(g(q)))<\bar\varepsilon
$$
and
$$
d\left(f^{n}(f(h(q))), j(g^{n+1}(q)\right)
=d\left(f^{n+1}(h(x)), j(g^{n+1}(q))\right)<\varepsilon\left(f^{n+1}(h(q))\right)
$$
for all $n \in \mathbb{Z}$, we know that two points $h(g(q))$ and $f(h(q))$  $\bar\varepsilon$-trace the $\delta$-pseudo orbit 
$(g^{n+1}(q))_{n \in \mathbb{Z}}$. By the uniqueness, we have $f \circ h=h \circ g$.

Now, we show $h$ is continuous for each $x_0 \in B(f(x),1/\delta)\subset X$.

Indeed, fix $\lambda > 0$. By lemma \ref{lem:3.3lee}, there is $N \in \mathbb N^+$  such that for any $x_1\in X$ if
$$
d(f^n(h(x_0)),f^n(h(x_1))) \le e(f^n(h(x_0)))
$$
for all $|n|<N$, then $d(h(x_0),h(x_1))<\lambda$.

Because $g$ is continuous and $\overline { B(y,1/\delta) }\subset Y$ is compact we have that 
$g$ is uniformly continuous on $\overline { B(y,1/\delta)} $. So, for $\eta>0$ such that 
$d(x_0,x_1) < \eta$ then
 $$
 d(g^n(x_0),g^n(x_1) < \frac{\gamma(h(g^n(x_0)))}{4}
 $$ 
 for all $|n|< N$. 

Then if $d(x_0, x_1) < \eta$ we have 
$$
\begin{aligned}
d^{X}&\left(f^{n}(h(x_0)), f^{n}\left(h\left(x_1\right)\right)\right)\\
=& d^{X}\left(h\left(g^{n}(x_0)\right), h\left(g^{n}\left(x_1\right)\right)\right) \\
\leq & d^{X}\left(h\left(g^{n}(x_0)\right), j\left(g^{n}(x_0)\right)\right)+
 d^{X}\left(j\left(g^{n}(x_0)\right), j\left(g^{n}\left(x_1\right)\right)\right)\\
&+ d^{X}\left(h\left(g^{n}\left(x_0\right)\right), j\left(g^{n}\left(x_1\right)\right)\right) \\
\leq &  \bar{\varepsilon}(hg^n(x_0))+\frac{\gamma(hg^n(x_0))}{4}+  \bar{\varepsilon}(hg^n(x_1))\\
< &  \frac{\gamma(hg^n(x_0))}{4}+\frac{\gamma(hg^n(x_0))}{4}+\frac{\gamma(hg^n(x_0))}{4}\\
< & \max\{ \gamma(f^n(h(x_0))),\gamma(f^n(h(x_1)))\}
< e (f^n(h(x_0))) \quad \text { by Remark \ref{rem:3.2lee}}
\end{aligned}
$$
$\forall |n| < N$, 
and so $d\left(h(x_0), h\left(x_1\right)\right) <\lambda$ by the choice of $N$. Then, $h$ is continuous.
\end{proof}

\textbf{Note.} In general, if $X_n$ is closed for all $n$, then $X_n\overset{GH}{\longrightarrow}X$ and $X$ be non-closed. However, $d_{GH}(X,\overline{X})=0$. 

\begin{exam}
Let $T_A:\mathbb{T}^2\to\mathbb{T}^2$ be the hyperbolic torus automorphism induced by the matrix
$$
A=\left(\begin{array}{cc}
    2 & 1 \\
    1 & 1
\end{array}
\right)
$$
on the torus $\mathbb{T}^2=\mathbb{S}^1\times\mathbb{S}^1$. It is knows that $T_A$ has the POTP and $T_A$ is expansive. Then, $T_A$ is $GH$-stable. In particular, $T_A$ is $pGH$-stable.

\end{exam}

\begin{exam}
\label{excherry}
Let $\mathcal{X}$ be a $C^{\infty}$ vector field which induces a Cherry flow $\varphi_t$ on
$\mathbb{T}^2$,  the two-dimensional torus (\cite{boyd85}, \cite{mepa82}).
\end{exam}

\begin{figure}
    \centering
    \includegraphics[scale=0.7]{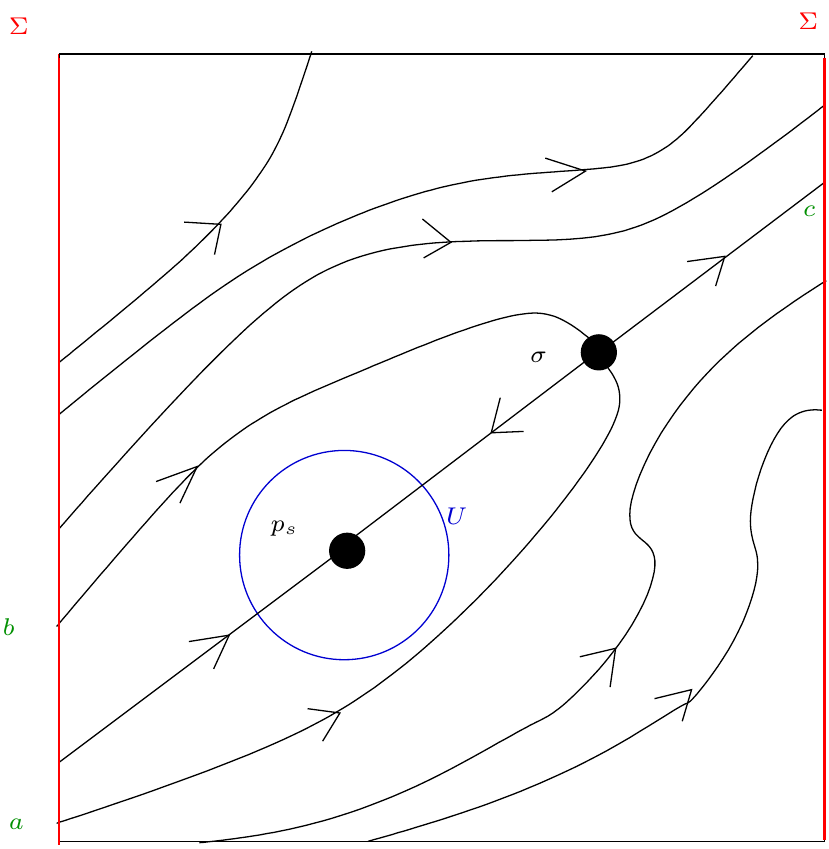}
    \caption{Cherry flow}
    \label{cherry1}
\end{figure}

The vector field $\mathcal{X}$ of the Cherry flow $\varphi_t$ has the following properties:

\begin{enumerate}
    \item $\mathcal{X}$ has two hyperbolic singularities, a saddle $\sigma$ and a sink  $p_s$.
    \item $\mathcal{X}$ is transverse to a meridian circle $\Sigma$ in $\mathbb{T}^2$.
    \item One of the two orbits in $W^u(\sigma)\setminus{\sigma}$ intersect $\Sigma$ in a first point $c$.  
    \item There is an open interval $(a,b)\subset\Sigma$ such that the positive orbit of $y\in (a,b)$ goes directly to $p_s$. 
    \item A Poincare map $g:\Sigma\setminus [a,b]\to\Sigma$ associated to $\mathcal{X}$ is expanding.
    \item The map $g$ in (5.) is extended to the whole $\Sigma$ defining $g(y)=c$ for every $y\in [a,b]$. Moreover, $g$ has irrational rotation number.
\end{enumerate}

The proof of the following lemma can be found in \cite{mepa82}.

\begin{Lemma}
\label{Cherrylemma}
If $\varphi_t$ is a Cherry flow, then 
\begin{enumerate}
    \item $\varphi_t$ has no periodic orbits.
    \item $\Lambda=\mathbb{T}^2\setminus W^s(p_s)$ is a transitive set of $\varphi_t$. 
\end{enumerate}
\end{Lemma}

To continue, define $f_0:\mathbb{T}^2\to\mathbb{T}^2$ by $f_0(x)=\varphi_1(x)$. By Lemma \ref{Cherrylemma}, $f_0 $ has the $\mathcal{C}$-shadowing property, and property (5.) implies $g$ expanding, but $f_0$ is not. However $f_0$ is $pGH$-stable by property (1.) and Lemma \ref{Cherrylemma}.  
Now, take $\epsilon>0$ small and consider an open neighborhood $U$ of the sink $p$ such that $U$ not intersects $\Sigma$, diameter of $U$ is less than $\epsilon$ and $\bigcap_{n\geq 1}\left(f_0\right)^n(U)=\{p_s\}$. For each $n\in\mathbb{N}$, $\mathbb{T}^2$ is pulled in $U$ as it is shows in Figure \ref{deftorus}. The deformation generates an other metric spaces $X_n$, composed by a torus, a sphere and a circular cylinder of length $n$ and radio $\epsilon/n$ connecting them, such as in Figure \ref{deftorus}. Let $i_n:\mathbb{T}^2\to X_n$ be a homeomorphism which produces the deformation from $\mathbb{T}^2$ to $X_n$, with $i_n(p_s)$ in the sphere. 
Now, take $f_n:X_n\to X_n$ given by $f_n=i_n\circ f_0\circ i_n^{-1}$.
Observe that $(X_n,p_n)\xrightarrow{pGH} 
(X,p)$, where $(X,p)$ can be: a torus $(\mathbb{T}^2\setminus\{p_s\},p)$, $(S^2\setminus\{q_u\},p)$, or $(\mathbb{R},p)$, depending the ubication of points $p_n$ \cite[example 6.3]{fukaya1990}.  Now, we take $\tilde{X}=\overline{X}$. So,  $(X_n,p_n,f_n)\xrightarrow{d_{GH^0}^p} (\tilde{X},p,f)$. 

\begin{figure}[H]
    \centering
    \includegraphics[scale=0.65]{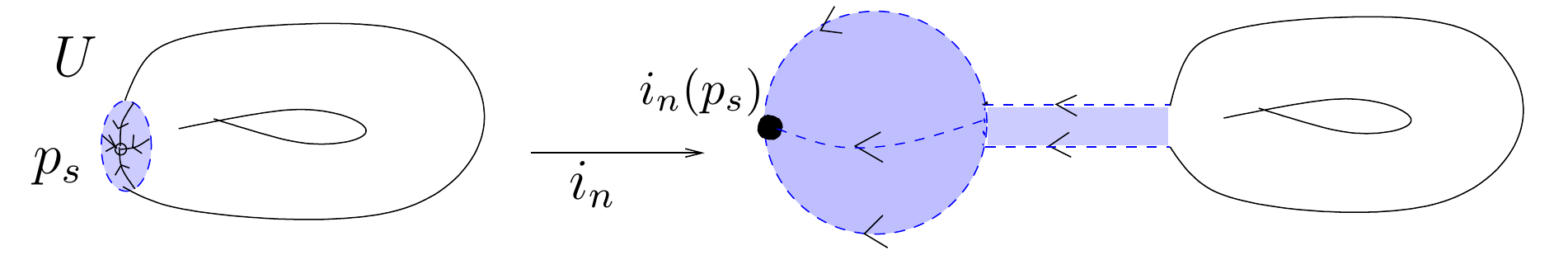}
\caption{Deforming torus}
    \label{deftorus}
\end{figure}

\begin{itemize}
    \item $f|_{\mathbb{T}^2}=f_0$.
    \item $f|_{\mathbb{R}}$ is an homeomorphism and has no fixed points.
    \item $f|_{S^2}$ has two singularities, a sink $q_s$ and a source $q_u$.
    \item If $p\in \mathbb{R}$, then $\alpha(p)=p_s$ and $\omega(p)=q_u$. Moreover, there exists a sequence $p_n\in X_n$ such that $(X_n,p_n)\xrightarrow{d^p_{GH^0}}(\mathbb{R
},p)$.
    \item If $p\in S^2\setminus\{q_u,q_s\}$, then $\alpha(p)=q_u$ and $\omega(p)=q_s$.
\end{itemize}

\begin{figure}[H]
    \centering
    \includegraphics[scale=0.8]{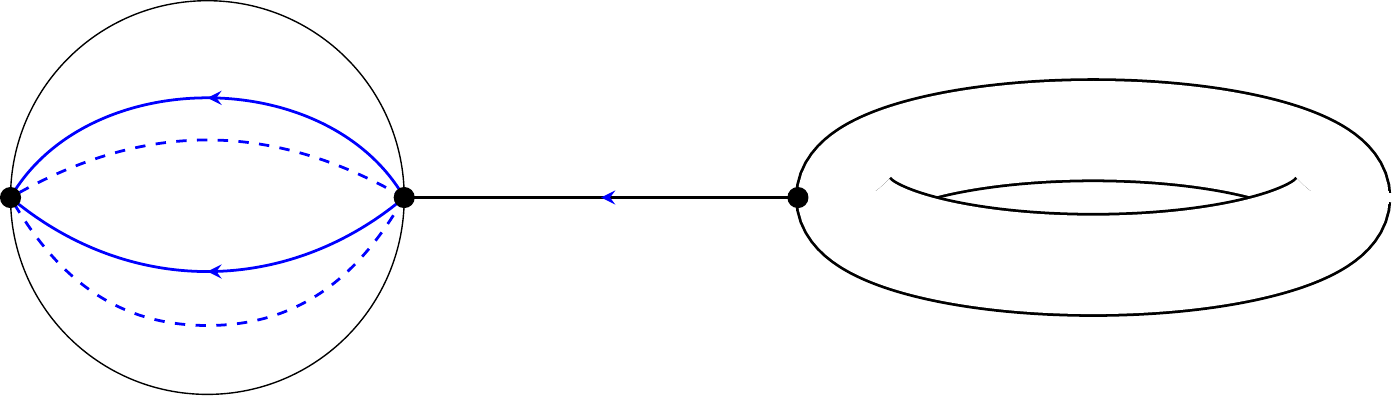}
    \caption{Dynamic in space $X$}
    \label{Cherry3}
\end{figure}

 In this way, $f$ is $pGH$-stable. 

\begin{exam}
Let $g:S^1\to S^1$ given by $g(\theta)=2\theta$, with $S^1=\mathbb{R}/2\pi\mathbb{Z}$ and $d(\theta_1,\theta_2)$ is a minimal arclength connecting $\theta_1$ and $\theta_2$ over $2\pi$. The homeomorphism $g$ is expansive and has the POTP \cite{devaney1989introduction}, so $g$ is $pGH$-stable. For each $n\in\mathbb{N}$, consider the homeomorphism $i_n:S^1\to Y_n$ which produces the deformation from $S^1$ to $Y_n$ such that a ball centered in $\theta=\pi$ and radio $1/4n$ is transformed such as is described in Figure \ref{circle1} in blue color section, where the neck has fixed length.  Define $f_n:Y_n\to Y_n$ by $f_n=i_n\circ g\circ i_n^{-1}$. From Theorem 2.3.6 in \cite{sakai96}  follows that for all $n$, $f_n$ has the POTP. On the other hand, the expansiveness is directly guaranteed by definition of $f_n$. Therefore, $f_n$ is $pGH$-stable for all $n$. In this case, $Y_n\xrightarrow{GH} Y$, where $Y$ is the union between two circles and a line connect them. However, the sequence $\{f_n\}_n$  no convergence to a map $f:Y\to Y$ in $GH^0$-sense neither $pGH^0$-sense. Indeed, taking points $p_n,q_n\in Y_n$ such as Figure \ref{circle1}, then $p_n\overset{d_{GH}^p}{\longrightarrow} x$ and $q_n\overset{d^p_{GH}}{\longrightarrow} x$ for some $x\in Y$ in the line. But $f(x)$ is not well defined. This example show us that in certain phenomena it is convenient to consider an other way to measure the distance between metric spaces and as consequence, distance between maps.

\begin{figure}[H]
    \centering
    \includegraphics{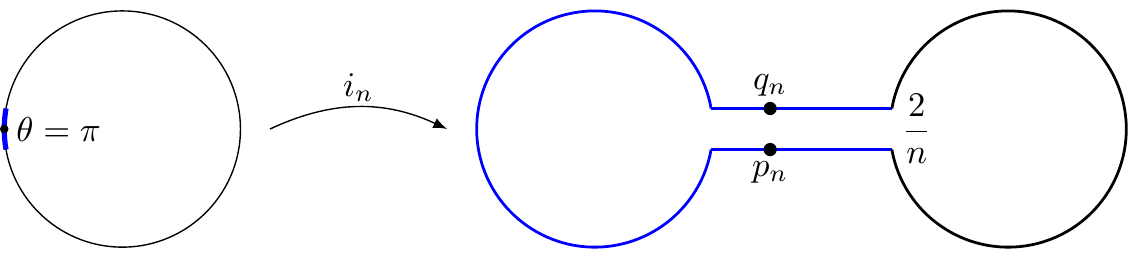}
    \caption{Deforming $S^1$}
    \label{circle1}
\end{figure}

\end{exam}

\section{Appendix}

\begin{Prop}[triangle inequality]\label{p:tinq1}
Let $\left(X_{m}, x_{m}\right)$ be pointed metric spaces for $m=1,2,3$. 
If  $d^p_{GH}\left((X_{m}, x_{m}),(X_{n}, x_{n})\right) \leq 1 / 2$ for $(m,n)=(1,2)$ and $(2,3)$,  then
\begin{align*}
d^p_{GH}\left((X_{1}, x_{1}),(X_{3}, x_{3})\right) \leq 
2\left[d^p_{GH}\left((X_{1}, x_{1}),(X_{2}, x_{2})\right)+
d^p_{GH}\left((X_{2}, x_{2}),(X_{3}, x_{3})\right)\right]
\end{align*}
\end{Prop}
\begin{proof}

 Let $i_{mn}\in App_{\varepsilon_{mn}}\left((X_{m}, x_{m}),(X_{n}, x_{n})\right)$, where $(m,n)$ equal to $(1,2)$, or $(2,3)$. We use the notation $B^{m}$ for $B^{X_m}$, and $d^{m}$ for $d^{X_m}$, and take $\varepsilon_{mn}<\frac12$.
Then we have that $i_{mn}(x_m)=x_n$, 
$$
dis(i_{mn})|_{B^m(x_m,\varepsilon_{mn}^{-1})}<\varepsilon_{mn},
\quad
B^n(x_n,\varepsilon_{mn}^{-1}-\varepsilon_{mn})\subset N_{\varepsilon_{mn}}\left(
i_{mn}(B^n(x_n,\varepsilon_{mn}^{-1}))\right).
$$
We call $i_{13} = i_{23} \circ i_{12}$, and let 
 $\varepsilon_{13}=2\left(\varepsilon_{12}+\varepsilon_{23}\right)$, we want to show that
$i_{13}\in App_{\varepsilon_{13}}\left((X_{1}, x_{1}),(X_{3}, x_{3})\right)$.

Observe that $\varepsilon_{mn}<2\varepsilon_{mn}< \varepsilon_{13}< 1$, 

Let $p\in B^{1}\left(x_{1}, \varepsilon_{13}^{-1}\right) \subset B^{1}\left(x_{1}, \varepsilon_{12}^{-1}\right)$, and since 
$dis(i_{12})|_{B^1(x_1,\varepsilon_{12}^{-1})}<\varepsilon_{12}$ and $i_{12}(x_1)=x_2$
we have that
\begin{align*}
 |d^1(p,x_1)-d^2(i_{12}(p),x_2)|\le \varepsilon_{12},
\end{align*}
then
\begin{align*}
 d^2(i_{12}(p),x_2)&\le d^1(p,x_1)+\varepsilon_{12}\\
& \le \varepsilon_{13}^{-1}+\varepsilon_{12}\\
& \le \frac{1+2\varepsilon_{12}(\varepsilon_{12}+\varepsilon_{23})}{2(\varepsilon_{12}+\varepsilon_{23})} \le \frac{1}{(\varepsilon_{12}+\varepsilon_{23})}\\
& \le \frac{1}{\varepsilon_{23}}.
\end{align*}
Note that we use the fact that  $\varepsilon_{mn}<1/2$.

Now if $p_1,p_2\in B^{1}\left(x_{1}, \varepsilon_{13}^{-1}\right)$, 
$i_{12}(p_1),i_{12}(p_2)\in B^{2}\left(x_{2}, \varepsilon_{23}^{-1}\right)$, then
\begin{align*}
 |&d^1(p_1,p_2)-d^3(i_{13}(p_1),i_{13}(p_2))|\\&
 \le|d^1(p_1,p_2)-d^2(i_{12}(p_1),i_{12}(p_2))|+
 |d^2(i_{12}(p_1),i_{12}(p_2))-d^3(i_{13}(p_1),i_{13}(p_2))|\\
& \le \varepsilon_{12}+\varepsilon_{23}<\varepsilon_{13}.
\end{align*}
Then $dis(i_{13})|_{B^1(x_1,\varepsilon_{13}^{-1})}<\varepsilon_{13}$.

For the third part,  let $p_3\in B^{3}\left(x_{3}, \varepsilon_{13}^{-1}-\varepsilon_{13}\right)\subset B^{3}\left(x_{3}, \varepsilon_{23}^{-1}\right)\subset N_{\varepsilon_{23}}\left(
i_{23}B^2(x_2,\varepsilon_{23}^{-1})\right)$, 
then there exist $p_2\in B^2(x_2,\varepsilon_{23}^{-1})$, such that $d^3(i_{23}(p_2),p_3)<\varepsilon_{23}$, and using the distortion of $i_{23}$
$$ 
\begin{aligned}
d^{2}\left(p_{2}, x_{2}\right) & \leq d^{3}(i_{23}(p_{2}), x_{3})+\varepsilon_{23} \\
& \leq d^{3}(i_{23}(p_{2}), p_{3}) +d^{3}(p_{3},x_3)+\varepsilon_{23}\\
& \leq \varepsilon_{13}^{-1}-\varepsilon_{13}+2\varepsilon_{23}
=\varepsilon_{13}^{-1}-2\varepsilon_{12}\\
& \leq \varepsilon_{12}^{-1}-\varepsilon_{12},
\end{aligned}
$$
that is, $p_{2} \in B^{2}\left(x_{2}, \varepsilon_{12}^{-1}-\varepsilon_{12}\right)
\subset N_{\varepsilon_{12}}\left(
i_{12}B^1(x_1,\varepsilon_{12}^{-1})\right)$, then exists $p_{1} \in B_{1}\left(x_{1},  \varepsilon_{12}^{-1}\right)$ such that $d^{2}\left(i_{12}\left(p_{1}\right), p_{2}\right)<\varepsilon_{12},$
and again by the dilation of $i_{12}$
$$
\begin{aligned}
d^{1}\left(p_{1}, x_{1}\right) & \leq d^{2}(i_{12}(p_{1}), p_{2})+d^{2}(p_{2}, x_{2})+\varepsilon_{12} \\
&<\varepsilon_{12}+\varepsilon_{13}^{-1}-\varepsilon_{13}+2\varepsilon_{23}+\varepsilon_{12}\\
&= \varepsilon_{13}^{-1},
\end{aligned}
$$
that is $p_1\in B^1(x_1,\varepsilon_{13}^{-1})$. 
Now using the distortion of $i_{23}$ we have
$$
\begin{aligned}
d^{3}\left(i_{13}\left(p_{1}\right), p_{3}\right) & \leq d^{3}\left(i_{13}\left(p_{1}\right), i_{23}\left(p_{2}\right)\right)+d^{3}\left(i_{23}\left(p_{2}\right), p_{3}\right) \\
& = d^{3}\left(i_{23}(i_{12}(p_{1})), i_{23}(p_{2})\right)+d^{3}\left(i_{23}(p_{2}), p_{3}\right)\\
& \leq d^{2}\left(i_{12}(p_{1}), p_{2}\right)+\varepsilon_{23}+d^{3}\left(i_{23}(p_{2}), p_{3}\right) \\
& \leq \varepsilon_{12}+2 \varepsilon_{23} \\
&<\varepsilon_{13}.
\end{aligned}
$$
That is $B^3(x_3,\varepsilon_{13}^{-1}-\varepsilon_{13})\subset N_{\varepsilon_{13}}(i_{13}(B^1(x_1,\varepsilon_{13}^{-1})))$.

Making $(m,n)$ equal to $(2,1)$, or $(3,2)$, and 
interchanging  $1$ by $3$, we get that  the same result for $i_{31}$, and the proof is complete.
\end{proof}

\end{document}